\documentclass[12pt]{amsart}
\usepackage{a4wide}
\usepackage[utf8]{inputenc}
\usepackage{amsmath}
\usepackage{amssymb}
\usepackage{amsfonts}
\usepackage{amsopn}
\usepackage{graphicx}
\usepackage{enumerate}
\usepackage{color}
\usepackage{mathtools}
\usepackage{mathrsfs}
\usepackage{bbm}
\usepackage{yhmath}
\usepackage{cite}
\usepackage[colorlinks,linkcolor=blue,anchorcolor=blue,citecolor=blue]{hyperref}
\usepackage{soul}

\newtheorem{theorem}{Theorem}[section]
\newtheorem{proposition}[theorem]{Proposition}
\newtheorem{lemma}[theorem]{Lemma}
\newtheorem{definition}[theorem]{Definition}
\newtheorem{corollary}[theorem]{Corollary}
\theoremstyle{definition}
\newtheorem*{question}{Question}
\newtheorem{example}[theorem]{Example}
\theoremstyle{remark}
\newtheorem*{remark}{Remark}

\numberwithin{equation}{section}

\newcommand{\mbf}{\mathbf}
\newcommand{\mcal}{\mathcal}

\newcommand{\set}[1]{\left\{ #1 \right\}}

\newcommand{\C}{\mathbb{C}}
\newcommand{\R}{\mathbb{R}}

\newcommand{\Z}{\mathbb{Z}}
\newcommand{\N}{\mathbb{N}}
\newcommand{\PP}{\mathbb{P}}

\newcommand{\f}{\infty}
\newcommand{\la}{\langle}
\newcommand{\ra}{\rangle}
\newcommand{\wh}[1]{\widehat{#1}}
\newcommand{\wt}[1]{\widetilde{#1}}

\newcommand{\ep}{\varepsilon}
\newcommand{\sse}{\subseteq}
\newcommand{\sm}{\setminus}
\newcommand{\D}{\;\mathrm{d}}

\newcommand{\blue}[1]{{\color{blue}#1}}

\title[Spectrality of Random Convolutions]{Spectrality of Random Convolutions Generated by Finitely Many Hadamard Triples}

\author[W. Li]{Wenxia Li}
\address[W. Li]{School of Mathematical Sciences, Key Laboratory of MEA (Ministry of Education) \& Shanghai Key Laboratory of PMMP, East China Normal University, Shanghai 200241, People's Republic of China}
\email{wxli@math.ecnu.edu.cn}

\author[J. J. Miao]{Jun Jie Miao}
\address[J. J. Miao]{School of Mathematical Sciences, Key Laboratory of MEA (Ministry of Education) \& Shanghai Key Laboratory of PMMP, East China Normal University, Shanghai 200241, People's Republic of China}
\email{jjmiao@math.ecnu.edu.cn}

\author[Z. Wang]{Zhiqiang Wang*}
\address[Z. Wang]{School of Mathematical Sciences, Key Laboratory of MEA (Ministry of Education) \& Shanghai Key Laboratory of PMMP, East China Normal University, Shanghai 200241,
People's Republic of China}
\email{zhiqiangwzy@163.com}

\thanks{* Corresponding author}

\begin{document}

\begin{abstract}
Let $\{(N_j, B_j, L_j): 1 \le j \le m\}$ be finitely many Hadamard triples in $\mathbb{R}$.
Given a sequence of positive integers $\{n_k\}_{k=1}^\infty$ and $\omega=(\omega_k)_{k=1}^\infty \in \{1,2,\cdots, m\}^\mathbb{N}$, let $\mu_{\omega,\{n_k\}}$ be the infinite convolution given by
$$\mu_{\omega,\{n_k\}} = \delta_{N_{\omega_1}^{-n_1} B_{\omega_1}} * \delta_{N_{\omega_1}^{-n_1} N_{\omega_2}^{-n_2} B_{\omega_2}} * \cdots * \delta_{N_{\omega_1}^{-n_1} N_{\omega_2}^{-n_2} \cdots N_{\omega_k}^{-n_k} B_{\omega_k} }* \cdots. $$
In order to study the spectrality of $\mu_{\omega,\{ n_k\}}$, we first show the spectrality of general infinite convolutions generated by Hadamard triples under the equi-positivity condition.
Then by using the integral periodic zero set of Fourier transform we show that if $\mathrm{gcd}(B_j - B_j)=1$ for $1 \le j \le m$, then all infinite convolutions $\mu_{\omega,\{n_k\}}$ are spectral measures.
This implies that we may find a subset $\Lambda_{\omega,\{n_k\}}\subseteq \mathbb{R}$ such that $\big\{ e_\lambda(x) = e^{2\pi i \lambda x}: \lambda \in \Lambda_{\omega,\{n_k\}} \big\}$ forms an orthonormal basis for $L^2(\mu_{\omega,\{ n_k\}})$.
\end{abstract}


\maketitle

\section{Introduction}

Let $\mcal{P}(\R^d)$ denote the set of all Borel probability measures on $\R^d$.
We call $\mu \in \mcal{P}(\R^d)$ a {\it spectral measure} if there exists a countable subset $\Lambda \sse \R^d$ such that the family of exponential functions
$$\set{ e_\lambda(x) = e^{2\pi i \lambda \cdot x}: \lambda \in \Lambda}
$$
forms an orthonormal basis in $L^2(\mu)$. The set $\Lambda$ is called a \textit{spectrum} of $\mu$, and we say that $(\mu,\Lambda)$ is a \textit{spectral pair}.
The existence of spectrum of $\mu$ is a basic question in harmonic analysis, and it may date back to Fuglede's seminal paper \cite{Fuglede-1974}. Fuglede conjectured that if $\Gamma \sse \R^d$ is a measurable subset with positive finite Lebesgue measure, the normalized Lebesuge measure on $\Gamma$ is a spectral measure if and only if $\Gamma$ tiles $\R^d$ by translations. Tao \cite{Tao-2004} and the others \cite{Farkas-Matolcsi-Mora-2006,Farkas-Revesz-2006,Kolountzakis-Matolcsi-2006a,Kolountzakis-Matolcsi-2006b,Matolcsi-2005}
have disproved Fuglede's conjecture in both directions for $d\ge 3$. However, the connection between spectrality and tiling still attracts considerable attentions, and some interesting positive results have been proved for special cases \cite{Iosevich-Katz-Tao-2003,Laba-2001,Lev-Matolcsi-2019}.

In 1998, Jorgensen and Pedersen \cite{Jorgensen-Pedersen-1998} showed that  $\mu_{4,\set{0,2}}$ is a spectral measure with a spectrum $\Lambda$,
where $\mu_{4,\set{0,2}}$ is the self-similar measure with equal weights generated by the iterated function system (IFS) $$\set{ f_1(x) =\frac{x}{4}, f_2(x) = \frac{x+2}{4} },$$ and the set $\Lambda$ is given by
\begin{equation}\label{lambda}
  \Lambda = \bigcup_{n=1}^\f \set{\ell_1 + 4\ell_2 + \cdots + 4^{n-1} \ell_n: \ell_1,\ell_2,\cdots,\ell_n \in \set{0,1}},
\end{equation}
but the standard middle-third Cantor measure is not a spectral measure.
Note that the fractal measure $\mu_{4,\set{0,2}}$ and the Lebesgue measure are mutually singular.
We refer the readers to \cite{Falco03} for the details of fractal sets and measures. This aroused the research interest of spectrality of fractal measures. From then on, there are abundant literatures on this topic \cite{An-Fu-Lai-2019,An-He-He-2019,An-He-2014,An-He-Lau-2015,An-He-Li-2015,An-Wang-2021,Dai-He-Lau-2014, Deng-Chen-2021,Dutkay-Han-Sun-2014,Dutkay-Haussermann-Lai-2019,Dutkay-Jorgensen-2007,Dutkay-Jorgensen-2012, Dutkay-Lai-2014,Dutkay-Lai-2017,Fu-Wen-2017,He-Tang-Wu-2019,Jorgensen-Pedersen-1998,Laba-Wang-2002, Strichartz-2000,Strichartz-2006,Wang-Dong-Liu-2018}.


There are many surprising phenomena for singular continuous spectral measures. In \cite{Dutkay-Jorgensen-2012}, Dutkay and Jorgensen showed that besides the set $\Lambda$ defined in (\ref{lambda}), the sets  $5\Lambda$, $7\Lambda$, $11\Lambda$, $13\Lambda$, $17\Lambda$, $\cdots$ are all spectra of $\mu_{4,\set{0,2}}$.
Moreover, the convergence of mock Fourier series is distinct with respect to different spectra.
In \cite{Strichartz-2006}, Strichartz proved that the mock Fourier series of continuous functions converges uniformly with respect to the spectral measure $\mu_{4,\set{0,2}}$ with the spectrum $\Lambda$, but Dutkey, Han and Sun \cite{Dutkay-Han-Sun-2014} showed that there exists a continuous function such that its mock Fourier series is divergent at $0$ with respect to the spectrum $17\Lambda$.

To study fractal spectral measures, the Hadamard triple is a fundamental tool. We write $\#$ for the cardinality of a set. Let $R\in M_d(\Z)$ be a $d\times d$ expanding matrix (i.e. all eigenvalues have modulus strictly greater than $1$) with integral entries. Let $B,L\subseteq \Z^d$ be two finite subsets of integral vectors with $N=\#B=\#L \ge 2$. If the matrix
 $$\left[ \frac{1}{\sqrt{N}} e^{-2\pi i (R^{-1}b) \cdot \ell} \right]_{b\in B,\ell \in L}$$
 is unitary, we call $(R, B, L)$ a {\it Hadamard triple} in $\R^d$ . We write $\delta_a$ for the Dirac measure at the point $a \in \R^d$, and for a finite subset $A\sse \R^d$, we write
 $$\delta_{A} = \frac{1}{\# A} \sum_{a \in A} \delta_a.$$
It is clear that $(R, B, L)$ is a Hadamard triple if and only if the set $L$ is a spectrum of the discrete measure $\delta_{R^{-1}B}$,

To construct more flexible examples of fractal spectral measures, different Hadamard triples are involved in infinite convolution.
The following question is largely investigated.
\begin{question}
  Given a sequence of Hadamard triples $\set{(R_j,B_j,L_j): j \ge 1}$ in $\R^d$, under what conditions is the infinite convolution $$\mu=\delta_{R_1^{-1} B_1} * \delta_{(R_2 R_1)^{-1} B_2} * \cdots * \delta_{(R_n \cdots R_2 R_1)^{-1} B_n} * \cdots$$ a spectral measure ?
\end{question}
Many affirmative results have been obtained in \cite{An-Fu-Lai-2019,An-He-He-2019,An-He-2014,An-He-Lau-2015,An-He-Li-2015,Dutkay-Haussermann-Lai-2019,Dutkay-Lai-2017,
Fu-Wen-2017,Laba-Wang-2002,Wang-Dong-Liu-2018}.
When all Hadamard triples are the same one, the infinite convolution reduces to a self-affine measure (which is called a self-similar measure for $d=1$).
{\L}aba and Wang \cite{Laba-Wang-2002} proved that the self-similar measure with equal weights generated by a Hadamard triple in $\R$ is a spectral measure, and Dutkay, Haussermann and Lai \cite{Dutkay-Haussermann-Lai-2019} generalized it to self-affine measures in higher dimension.

In this paper, we explore the spectrality of the random convolutions generated by finitely many Hadamard triples in $\R$.
Let $\set{(N_j, B_j, L_j): 1 \le j \le m}$ be finitely many Hadamard triples in $\R$.
Let $\Omega$ be the symbolic space over the alphabet $\set{1,2,\cdots, m}$.
Given  a sequence of positive integers $\{n_k\}_{k=1}^\f$ and $\omega=(\omega_k)_{k=1}^\f \in \Omega$, let $\mu_{\omega,\set{n_k}}$ be the random convolution given by
\begin{equation}\label{mu-subsequence}
  \mu_{\omega,\set{n_k}} = \delta_{N_{\omega_1}^{-n_1} B_{\omega_1}} * \delta_{N_{\omega_1}^{-n_1} N_{\omega_2}^{-n_2} B_{\omega_2}} * \cdots * \delta_{N_{\omega_1}^{-n_1} N_{\omega_2}^{-n_2} \cdots N_{\omega_k}^{-n_k} B_{\omega_k} }* \cdots ,
\end{equation}
where $\omega_k$ determines the Hadamard triple chosen in $k$-th convolution.
When $n_k =1$ for all $k \ge 1$, we write
\begin{equation}\label{mu-omega}
  \mu_\omega = \delta_{N_{\omega_1}^{-1} B_{\omega_1}} * \delta_{(N_{\omega_1} N_{\omega_2})^{-1} B_{\omega_2}} * \cdots * \delta_{(N_{\omega_1} N_{\omega_2}\cdots N_{\omega_k})^{-1} B_{\omega_k} }* \cdots.
\end{equation}

 The random convolution generated by finitely many Hadamard triples in $\R^d$ was first studied by Strichartz in \cite{Strichartz-2000}, where he showed that the random convolution is spectral under certain separation condition. But, in general, the separation condition is not easy to check.
 Dutkay and Lai also considered this question in \cite{Dutkay-Lai-2017}, and they generalized the Strichartz criterion under no-overlap condition and showed that for some special cases, there exists a common spectrum for almost all random convolutions.

In this paper, under a simple condition on the sets $\set{B_j: 1\le j \le m}$ in $\R$, we show that all random convolutions are spectral measures.

\begin{theorem}\label{main-result}
  Let $\set{(N_j, B_j, L_j): 1 \le j \le m}$ be finitely many Hadamard triples in $\R$.
  Suppose that $\gcd(B_j - B_j) =1$ for $1\le j \le m$.
Given a sequence $\set{n_k}_{k=1}^\f$ of positive integers, let $\mu_{\omega,\set{n_k}}$ be given by \eqref{mu-subsequence}. Then $\mu_{\omega,\set{n_k}}$ is a spectral measure for every $\omega \in \Omega$.
\end{theorem}
\begin{remark}
  (i) Since this manuscript was completed, the referee mentioned the independent work of Lu, Dong and Zhang \cite{Lu-Dong-Zhang-2022}, where they also obtained this result in a slight different form.
  In order to prove Theorem \ref{main-result}, we establish a general result (see Theorem \ref{general-result}), which can be applied to construct singular continuous spectral measures without compact support in the later work.

  (ii) An, Fu and Lai \cite{An-Fu-Lai-2019} studied the spectrality of infinite convolutions generated by a sequence of Hadamard triples $\{(N_n,B_n,L_n): n \ge 1\}$ in $\R$ under the restriction that $N_n \ge 2$ and $B_n \subseteq \{0,1,\cdots, N_n-1\}$. By analyzing equi-positivity and admissibility of probability measures, they proved that if $\liminf_{n\to \f} \# B_n < \f$, then the resulting infinite convolution is a spectral measure.
  Since they assumed $B_n \subseteq \{0,1,\cdots, N_n-1\}$, the infinite convolutions are all supported in $[0,1]$.
  They also pointed out that the admissibility of probability measures supported outside $[0,1]$ cannot be easily analyzed in their approach.
  The infinite convolutions we considered in Theorem \ref{main-result} may be supported outside $[0,1]$.
  In order to analyze the admissibility we generalize the argument in \cite[Theorem 5.4]{Dutkay-Haussermann-Lai-2019} for self-similar measures to more general infinite convolutions.
\end{remark}

As a special case of Theorem \ref{main-result}, we have the following corollary immediately.
\begin{corollary}\label{main-cor}
  Let $\set{(N_j, B_j, L_j): 1 \le j \le m}$ be finitely many Hadamard triples in $\R$.
  Suppose that $\gcd(B_j - B_j) =1$ for $1\le j \le m$.
Let $\mu_{\omega}$ be given by \eqref{mu-omega}. Then $\mu_{\omega}$ is a spectral measure for every $\omega \in \Omega$.
\end{corollary}

\blue{

}

To prove Theorem \ref{main-result}, we need to consider the infinite convolution generated by a sequence of Hadamard triples $\set{(N_n,B_n,L_n): n \ge 1}$ in $\R$.
We write
\begin{equation}\label{mu-n}
  \mu_n= \delta_{N_{1}^{-1} B_1} * \delta_{(N_1 N_2)^{-1} B_2} * \cdots * \delta_{(N_1 N_2 \cdots N_n)^{-1} B_n }.
\end{equation}
\emph{We always assume that $\mu_n$ converges weakly to a Borel probability measure $\mu$.}
It is known that if
\begin{equation}\label{condition-weak-converge}
  \sum_{n=1}^{\f} \frac{\max\set{|b|: b \in B_n}}{|N_1 N_2 \cdots N_n|} < \f,
\end{equation}
then $\mu_n$ converges weakly to a Borel probability measure $\mu$, and moreover, the measure $\mu$ has a compact support.
Noting that there are only finitely many Hadamard triples in Theorem~\ref{main-result}, it is obvious that  the condition (\ref{condition-weak-converge}) is satisfied.

The weak limit measure $\mu$ may be written as an infinite convolution
\begin{equation}\label{mu-infinite-convolution}
\begin{split}
  \mu & = \delta_{N_{1}^{-1} B_1} * \delta_{(N_1 N_2)^{-1} B_2} * \cdots * \delta_{(N_1 N_2 \cdots N_n)^{-1} B_n } * \cdots \\
  & = \mu_n * \mu_{>n}.
\end{split}
\end{equation}
We scale the measure $\mu_{>n}$ and define
\begin{equation}\label{nu-large-than-n}
  \nu_{>n}(\;\cdot\;) = \mu_{>n}\left( \frac{1}{N_1 N_2 \cdots N_n} \; \cdot\; \right).
\end{equation}
Obviously, the spectrality of $\mu$ is  affected by the property of the sequence $\{\nu_{>n}\}$.
The equi-positivity condition plays an important role in the study of spectrality.
\begin{definition}\label{def-equipositive}
  We call $\Phi \sse \mcal{P}(\R^d)$ an equi-positive family if there exist $\ep>0$ and $\delta>0$ such that for  $x\in [0,1)^d$ and $\mu\in \Phi$, there exists an integral vector $k_{x,\mu} \in \Z^d$ such that
  $$ |\wh{\mu}(x+y+k_{x,\mu})| \ge \ep,$$
  for all $ |y| <\delta,$ where $k_{x,\mu} ={\bf 0}$ for $x={\bf 0}$.
\end{definition}
The equi-positivity condition was first used in \cite{Dutkay-Haussermann-Lai-2019} for the spectrality of self-affine measures.
Our definition is more general than the one in~\cite{An-Fu-Lai-2019} since we do not assume that
all supports of probability measures in $\Phi$ are contained in a common compact subset.
The two definitions are equivalent for the tight family of probability measures, see Section~\ref{sec_pre} for the definition of tightness.

The following theorem is a generalization of Theorem 1.2 in~\cite{Dutkay-Lai-2017} and Theorem 3.2 in~\cite{ An-Fu-Lai-2019} for $d= 1$, and it is key to prove Theorem \ref{main-result}.

\begin{theorem}\label{general-result}
  Let $\set{(N_n, B_n , L_n): n \ge 1}$ be a sequence of Hadamard triples in $\R$. Let the probability measure $\mu$ be the weak limit of $\mu_n$ given by \eqref{mu-n}.
  If there exists a subsequence $\{ n_j \}$ of positive integers such that the family $\{ \nu_{>n_j} \}$ is equi-positive, then $\mu$ is a spectral measure with a spectrum in $\Z$.
\end{theorem}

Recently, Deng and Li \cite{Deng-Li-2023} showed that a class of Moran-type Bernoulli convolutions not generated by a sequence of Hadamard triples can also be spectral measures.
This implies that any assumption in Theorem \ref{main-result} is not necessary for spectrality.
At the end, we give an example to show that our assumption  $\gcd(B_j - B_j) =1$ for $1\le j \le m$ in Theorem \ref{main-result} is necessary to conclude that all random convolutions are spectral measures.
\begin{example}
  Let $N_1 = N_2 = 2$, $B_1 = \set{0,1}$, $B_2 = \set{0,3}$, $L_1 = L_2=\set{0,1}$.
  We have that $(N_i, B_i, L_i)$ is a Hadamard triple for $i=1,2$, but $\gcd(B_2 - B_2) =3 \ne 1$.
  Let $\eta = 1 2^\f$, and the random convolution
  $$\mu_\eta = \frac{1}{3} \mathscr{L}|_{[0,1/2]} + \frac{2}{3} \mathscr{L}|_{[1/2,3/2]} + \frac{1}{3} \mathscr{L}|_{[3/2,2]}, $$
  where $\mathscr{L}$ denotes the Lebesgue measure in $\R$.

  It has been showed that an absolutely continuous spectral measure must be uniform on its support \cite{Dutkay-Lai-2014}.
  It is clear that $\mu_\eta$ is not uniformly distributed on its support $[0,2]$.
  Thus, $\mu_\eta$ cannot be a spectral measure.

On the other hand, by Theorem 1.5 in \cite{Dutkay-Lai-2017}, there exists a subset $\Lambda \sse \Z$ such that $\Lambda$ is a spectrum of $\mu_\omega$ for $\PP$-a.e. $\omega \in \Omega = \set{1,2}^\f$, where $\PP$ is the Bernoulli measure with equal weights on $\Omega$.
This implies that $\eta$ is in the exceptional set which is $\PP$-measure zero.
\end{example}

We organize our paper as follows: in Section~\ref{sec_pre}, we recall some definitions and review some already-known results which are essential in our proofs;
in Section~\ref{sec_ea}, we study properties of the equi-positivity and admissibility in $\R^d$, and we prove Theorem \ref{general-result};
we give the proof of Theorem \ref{main-result} in Section~\ref{sec_pf}.

\section{Preliminaries}\label{sec_pre}

First, we give some simple facts about symbolic spaces which are frequently used in our context.
Let $\Omega$ be the symbolic space over the alphabet $\set{1,2,\cdots,m}$, i.e.
$$
\Omega=\set{1,2,\cdots,m}^{\N}.
$$
We topologize the symbolic space $\Omega$ by the metric
$$ d(\omega,\eta) = 2^{-\min\set{k \ge 1: \omega_k \ne \eta_k} }
$$
for distinct $\omega=\omega_1\omega_2\cdots$, $\eta=\eta_1\eta_2\cdots  \in \Omega$. 
Then $\Omega$ becomes a compact metric space.
It is well-known that a sequence $\set{\omega(j)}_{j=1}^\f \sse \Omega$ converges to $\omega$ if and only if for each $k \ge 1$, there exists $j_0 \ge 1$ such that for every $j \ge j_0$, $$\omega_1(j) \omega_2(j) \cdots \omega_k(j) = \omega_1 \omega_2 \cdots \omega_k.$$
The left shift on $\Omega$ is denoted by $\sigma$, that is, for $\omega=\omega_1\omega_2\cdots \in \Omega$,
$$\sigma(\omega) = \omega_2 \omega_3 \cdots.$$

For $\mu \in \mcal{P}(\R^d)$, the \emph{support} of $\mu$ is defined by $$\mathrm{spt}(\mu) = \R^d \sm \bigcup \set{ U\sse \R^d: \text{ $U$ is open, and $\mu(U)=0$}},$$
i.e., the smallest closed subset with full measure.
For a compact subset $K \sse \R^d$, let $$\mcal{P}(K) = \set{ \mu\in \mcal{P}(\R^d): \mathrm{spt}(\mu) \sse K }. $$
For $\mu \in \mcal{P}(\R^d)$, the \emph{Fourier transform} of $\mu$ is defined by
$$ \wh{\mu}(\xi) = \int_{\R^d} e^{-2\pi i \xi \cdot x} \D \mu(x). $$
It is easy to verify that $\wh{\mu}$ is uniformly continuous and $\wh{\mu}(0) =1$.

Let $\mu,\mu_1,\mu_2,\cdots \in \mcal{P}(\R^d)$. Recall that  $\mu_n$ \textit{converges weakly} to $\mu$ if $$\lim_{n \to \f} \int_{\R^d} f(x) \D \mu_n(x) = \int_{\R^d} f(x) \D \mu(x),$$
for all $ f \in C_b(\R^d),$ where $C_b(\R^d)$ is the set of all bounded continuous functions on $\R^d$.

Let $\Phi \sse \mcal{P}(\R^d)$. We say that $\Phi$ is {\it tight} (sometimes called uniformly tight) if for each $\ep>0$ there exists a compact subset $K \sse \R^d$ such that $$\inf_{\mu \in \Phi} \mu(K) > 1 - \ep.$$
We refer the readers to \cite{Bil03} for more details about tightness. Given $\Phi \sse \mcal{P}(K)$ for some compact subset $K \sse \R^d$, it is clear that  $\Phi$ is tight.

Next we cite two well-known theorems which are frequently applied to weak limit of probability measures in our context.

\begin{theorem}\label{equivalent-condition-weak-converge}
  Let $\mu,\mu_1,\mu_2,\cdots \in \mcal{P}(\R^d)$. Then $\mu_n$ converges weakly to $\mu$ if and only if $\displaystyle \lim_{n \to \f} \wh{\mu}_n(\xi)=\wh{\mu}(\xi)$ for every $\xi \in \R^d$.
\end{theorem}

\begin{theorem}\label{weak-compactness}
  Let $\Phi \sse \mcal{P}(\R^d)$. Then $\Phi$ is tight if and only if, for every sequence $\{ \mu_n \} \sse \Phi$, there exists a subsequence $\{ \mu_{n_j} \}$ and a Borel probability measure $\mu \in \mcal{P}(\R^d)$ such that $\mu_{n_j}$ converges weakly to $\mu$ as $j \to \f$.
\end{theorem}

A family of continuous functions $\mcal{F} \sse C(\R^d)$ is called \textit{equicontinuous} if for each $\ep>0$ there exists $\delta >0$ such that $|f(x)-f(y)| < \ep$ for all $x,y\in \R^d$ satisfying $|x-y| <\delta$ and all $f\in \mcal{F}$.
The following lemma shows that the Fourier transforms of a tight family of probability measures is equicontinuous.

\begin{lemma}\label{equicontinuous}
  Let $\Phi \sse \mcal{P}(\R^d)$. If $\Phi$ is tight, then the family $\set{\wh{\mu}: \mu\in \Phi}$ is equicontinuous.
\end{lemma}
\begin{proof}
  For each $\ep>0$, since $\Phi$ is tight, there exists a compact subset $K \sse \R^d$ such that $$\inf_{\mu \in \Phi} \mu(K) > 1 - \frac{\ep}{3}.$$
  Then we may find $\delta >0$ such that for all $|y| < \delta$ and all $x\in K$, $$\left| 1- e^{2\pi i y\cdot x} \right| < \frac{\ep}{3}.$$
  Thus, for all $\mu \in \Phi$ and all $\xi_1,\xi_2 \in \R^d$ with $|\xi_1 - \xi_2| <\delta$, we have
  \begin{align*}
    \left| \wh{\mu}(\xi_1) - \wh{\mu}(\xi_2) \right| & = \left| \int_{\R^d} e^{-2\pi i \xi_1 \cdot x} \big(1- e^{ 2 \pi i (\xi_1 - \xi_2) \cdot x} \big) \D \mu(x) \right| \\
    & \le \int_K \left| 1- e^{ 2\pi i (\xi_1 - \xi_2) \cdot x} \right| \D \mu(x) + \int_{\R^d \sm K} \left| 1- e^{2\pi i (\xi_1 - \xi_2) \cdot x}  \right| \D \mu(x) \\
    & \le \frac{\ep}{3} \mu(K) + 2 \mu(\R^d \sm K) < \ep.
  \end{align*}
Therefore, the family $\set{\wh{\mu}: \mu\in \Phi}$ is equicontinuous.
\end{proof}

For $\mu,\nu \in \mcal{P}(\R^d)$, the convolution $\mu *\nu$ is given by   $$ \mu*\nu(B) = \int_{\R^d} \nu(B-x) \D \mu(x)= \int_{\R^d} \mu(B-y) \D \nu(y), $$
for every Borel subset $B\sse \R^d$.
Equivalently, the convolution $\mu *\nu$ is the unique Borel probability measure satisfying $$\int_{\R^d} f(x) \D \mu*\nu(x) = \int_{\R^d \times \R^d} f(x+y) \D \mu \times \nu(x,y),$$
for  all $ f \in C_b(\R^d).$
It is easy to check that $$\wh{\mu *\nu}(\xi) = \wh{\mu}(\xi) \wh{\nu}(\xi).$$
Using Theorem \ref{equivalent-condition-weak-converge}, it is straightforward to obtain the following lemma.

\begin{lemma}\label{convolution-weak-convergence}
  Let $\{\mu_n\}, \{\nu_n\} \sse \mcal{P}(\R^d)$.
  If $\mu_n$ and $\nu_n$ converge weakly to $\mu$ and $\nu$ respectively, then we have $\mu_n * \nu_n$ converges weakly to $\mu*\nu$.
\end{lemma}

To prove the spectrality of measures, we have to rely on the properties of Hadamard triples. We list some useful facts in the following lemma, see \cite{Laba-Wang-2002,Dutkay-Haussermann-Lai-2019} for details.
\begin{lemma}\label{lemma-HT}
  Let $(R,B,L)$ be a Hadamard triple in $\R^d$. Then we have

\noindent$(\mathrm{i})$ $(R,B+b_0, L+\ell_0)$ is also a Hadamard triple for all $b_0,\ell_0 \in \Z^d$;

\noindent$(\mathrm{ii})$ the elements in $B$ are in distinct residue classes modulo $R \Z^d$; the elements in $L$ are in distinct residue classes modulo $R^{\mathrm{T}} \Z^d$, where the superscript ${^\mathrm{T}}$ denotes the transpose of a matrix;

\noindent$(\mathrm{iii})$ if $\wt{L} \equiv L \pmod{ R^\mathrm{T}\Z^d}$, then $(R,B,\wt{L})$ is also a Hadamard triple;

\noindent$(\mathrm{iv})$ if $\set{(R_j, B_j, L_j): 1\le j \le n}$ are finitely many Hadamard triples in $\R^d$, let $$\mathbf{R}= R_n R_{n-1} \cdots R_1,\quad \mathbf{B} = (R_n R_{n-1} \cdots R_2) B_1 + \cdots + R_n B_{n-1} + B_n, $$ and $$\mathbf{L} = L_1 + R_1^{\mathrm{T}} L_2 + \cdots + (R_1^{\mathrm{T}} R_2^{\mathrm{T}} \cdots R_{n-1}^{\mathrm{T}}) L_n,$$
  then $(\mathbf{R},\mathbf{B},\mathbf{L})$ is a Hadamard triple.
\end{lemma}

The next theorem is often used to check whether a measure is spectral.
Although Jorgensen and Pedersen in~\cite{Jorgensen-Pedersen-1998} proved this theorem only for probability measures with compact support, it actually holds for all probability measures on $\R^d$.
For convenience, we give a detailed proof.
We use $U(x,\gamma)$ to denote the open ball centred at $x$ with radius $\gamma$ in $\R^d$.

\begin{theorem}\label{criterion}
  Let $\mu \in \mcal{P}(\R^d)$, and let $\Lambda \subseteq \R^d$ be a countable subset. We define $$Q(\xi) = \sum_{\lambda\in \Lambda} |\wh{\mu}(\lambda + \xi)|^2.$$

\noindent$(i)$ The family of exponential functions $\set{e_\lambda(x): \lambda \in \Lambda}$ forms an orthonormal set in $L^2(\mu)$ if and only if $Q(\xi) \le 1$ for all $ \xi \in \R^d$.

\noindent$(ii)$ The family of exponential functions $\set{e_\lambda(x): \lambda \in \Lambda}$ forms an orthonormal basis in $L^2(\mu)$ if and only if $Q(\xi) = 1$ for all $\xi \in \R^d$.
\end{theorem}

\begin{proof}
  Note that \[ \la e_\xi, e_\lambda \ra_{L^2(\mu)} = \int_{\R^d} e^{2\pi i (\xi - \lambda) \cdot x} \D \mu(x) = \wh{\mu}(\lambda - \xi). \]

  (i) For the necessity, note that $\{ e_\lambda: \lambda \in \Lambda \}$ is an orthonormal set in $L^2(\mu)$, and by the Bessel's inequality, for every $\xi \in \R^d$ we have
  \[ Q(\xi) = \sum_{\lambda\in \Lambda} |\wh{\mu}(\lambda + \xi)|^2 =  \sum_{\lambda\in \Lambda} \big| \la e_{-\xi}, e_\lambda \ra_{L^2(\mu)} \big|^2 \le \|e_{-\xi}\|_{L^2(\mu)}^2 =1.  \]
  For the sufficiency, take $\xi=-\lambda_0$ for some $\lambda_0 \in \Lambda$, and then we have
  \[ 1 \ge Q(-\lambda_0) = \sum_{\lambda \in \Lambda} |\wh{\mu}(\lambda-\lambda_0)|^2 = 1 + \sum_{\lambda \in \Lambda \sm \{\lambda_0\}} \big| \la e_{\lambda_0}, e_{\lambda} \ra_{L^2(\mu)} \big|^2. \]
  It follows that $\la e_{\lambda}, e_{\lambda'} \ra_{L^2(\mu)}=0$ for any $\lambda \ne \lambda' \in \Lambda$.
  This implies that $\{e_\lambda: \lambda\in \Lambda\}$ is an orthonormal set in $L^2(\mu)$.

  (ii) For the necessity, since $\{ e_\lambda: \lambda \in \Lambda \}$ ia an orthonormal basis in $L^2(\mu)$, by the Parseval's identity we have
  \[ Q(\xi) = \sum_{\lambda\in \Lambda} |\wh{\mu}(\lambda + \xi)|^2 =  \sum_{\lambda\in \Lambda} \big| \la e_{-\xi}, e_\lambda \ra_{L^2(\mu)} \big|^2 = \|e_{-\xi}\|_{L^2(\mu)}^2 =1.  \]

  In the following, we show the sufficiency.
  By (i) we have $\{e_\lambda: \lambda\in \Lambda\}$ is an orthonormal set in $L^2(\mu)$.
  Let $H$ be the closure of the spanning space of $\{ e_\lambda: \lambda \in \Lambda \}$ in $L^2(\mu)$.
  It suffices to show $H = L^2(\mu)$.
  Note that the space $C_c(\R^d)$ of continuous functions with compact support in $\R^d$ is dense in $L^2(\mu)$ \cite[Theorem 3.14]{Rudin-1987}.
  We only need to show \[ C_c(\R^d) \subseteq H. \]

  For $\xi\in \R^d$ we have
  \[ \sum_{\lambda\in \Lambda} \big| \la e_{\xi}, e_\lambda \ra_{L^2(\mu)} \big|^2= \sum_{\lambda\in \Lambda} |\wh{\mu}(\lambda - \xi)|^2 =1 = \|e_{\xi}\|_{L^2(\mu)}^2. \]
  Note that $\{e_\lambda: \lambda\in \Lambda\}$ is orthonormal in $L^2(\mu)$.
  We conclude that
  \begin{equation}\label{e-xi-contained}
    \{e_\xi(x) = e^{2\pi i \xi \cdot x}: \xi \in \R^d\} \subseteq H.
  \end{equation}

  Fix $f(x) \in C_c(\R^d)$.
  For $k \ge 1$, we define \[ f_k(x) = \sum_{n \in \Z^d} f(x + kn). \]
  Then $f_k(x)$ is a continuous $k\Z^d$-periodic function.
  Therefore, by the Stone-Weierstrass theorem, 
  the function $f_k(x)$ can be uniformly approximated under the supremum norm by functions in the spanning space of $\{ e^{2\pi i n\cdot x / k}: n \in \Z^d \}$.
  Since $H$ is a closed linear subspace, by \eqref{e-xi-contained}, we have $f_k \in H$ for each $k \ge 1$.

  Let \[ M = \sup_{x\in \R^d} |f(x)|,\; \gamma = \sup\{ |x|: x\in \R^d, f(x) \ne 0 \}. \]
  When $k > 2\gamma$, for every $x \in \R^d$ the function $f$ has at most one nonzero value at the points $x+ kn, n \in \Z^d$. It follows that for $k > 2\gamma$, $|f_k(x)| \le M$ for every $x \in \R^d$.
  Moreover, when $k > \gamma$ we have $f_k(x) = f(x)$ for every $|x|< k-\gamma$.
  Thus, for $k > 2\gamma$ we have
  \begin{align*}
    \int_{\R^d} |f_k(x) - f(x)|^2 \D \mu(x) & = \int_{\R^d \setminus U({\bf 0},k-\gamma)} |f_k(x) - f(x)|^2 \D \mu(x) \\
    & \le 4M^2 \mu\big( \R^d \setminus U({\bf 0},k-\gamma) \big)\rightarrow 0.
  \end{align*}
  It follows that $f_k$ tends to $f$ in $L^2(\mu)$ as $k \to \f$.
  Since $\{f_k\} \subseteq H$ and $H$ is closed, we obtain $f\in H$.
  Since $f$ is arbitrary, we conclude that $C_c(\R^d) \subseteq H$. This completes the proof.
\end{proof}

\section{Equi-positivity and admissibility}\label{sec_ea}

\subsection{Equi-positivity}
First, we give the proof of Theorem \ref{general-result}, and it is inspired by~\cite{An-Fu-Lai-2019, Strichartz-2000}.
\begin{proof}[Proof of Theorem \ref{general-result}]
  By Lemma \ref{lemma-HT} (i), we may assume that $0\in L_n$ for $n\ge 1$.
  Since the family $\{\nu_{>n_j}\}$ is equi-positive, there exist $\ep>0$ and $\delta>0$ such that for $x\in [0,1)$ and $j \ge 1$, there exists an integer $k_{x,j} \in \Z$ such that
  $$ |\wh{\nu}_{>n_j}(x+y+k_{x,j})| \ge \ep,$$
  for all $|y| <\delta$,
  and $k_{x,j}=0$ for $x=0$.

  For integers $q > p \ge 0$, we define $\mbf{N}_{p,q} = N_{p+1} N_{p+2} \cdots N_q $, $$ \mbf{B}_{p,q}= N_{p+1} N_{p+2} \cdots N_q \left( \frac{B_{p+1}}{ N_{p+1}} + \frac{B_{p+2}}{N_{p+1} N_{p+2}} + \cdots + \frac{B_q}{ N_{p+1} N_{p+2} \cdots N_q} \right) $$
  and $$\mbf{L}_{p,q} = L_{p+1} + N_{p+1} L_{p+2} + \cdots + ( N_{p+1} N_{p+2}\cdots N_{q-1}) L_q.$$
  By Lemma \ref{lemma-HT} (iv),  $(\mbf{N}_{p,q}, \mbf{B}_{p,q}, \mbf{L}_{p,q})$ is a Hadamard triple.

  We construct a sequence of finite subsets $\Lambda_i \sse \Z$ for $i \ge 1$ by induction.
  Let $m_1 = n_1$ and $\Lambda_1 = \mbf{L}_{0, m_1}$. Note that $0\in \Lambda_1$ and $\Lambda_1$ is a spectrum of $\mu_{m_1}$.
  For $i \ge 2$, suppose that $\Lambda_{i-1}$ has been defined, and
  we choose a sufficiently large element $m_i$ in the sequence $\set{n_j}$ such that $m_i > m_{i-1}$ and for all $ \lambda \in \Lambda_{i-1}$,
  \begin{equation}\label{ineqla}
  \left| \frac{\lambda}{N_1 N_2 \cdots N_{m_i}} \right| < \frac{\delta}{2}.
  \end{equation}
  Now we define
\begin{equation} \label{defLam}
  \Lambda_i = \Lambda_{i-1} + \mbf{N}_{0,m_{i-1}} \set{ \lambda + k_{\lambda,i} \cdot \mbf{N}_{m_{i-1}, m_i}  : \lambda \in \mbf{L}_{m_{i-1}, m_i} },
\end{equation}
  where, by the equi-positivity of $\{ \nu_{>n_j} \}$, the integers $k_{\lambda,i}\in \Z$ are chosen to satisfy
  \begin{equation}\label{lowerbound}
    \left| \wh{\nu}_{>m_i}\left( \frac{\lambda}{ N_{m_{i-1}+1} N_{m_{i-1}+2} \cdots N_{m_i} } + y + k_{\lambda,i} \right) \right|\ge \ep,
  \end{equation}
  for all $ |y|<\delta$,  and $k_{\lambda,i} =0$ for $\lambda=0$.
  Note that $\mu_{m_{i-1}} = \delta_{\mbf{N}_{0,m_{i-1}}^{-1} \mbf{B}_{0,m_{i-1}}}$ is a spectral measure with a spectrum $\Lambda_{i-1}$, and
  $$\mu_{m_i} = \delta_{\mbf{N}_{0,m_{i-1}}^{-1} \mbf{B}_{0,m_{i-1}}} * \delta_{(\mbf{N}_{0,m_{i-1}} \mbf{N}_{m_{i-1}, m_i})^{-1} \mbf{B}_{m_{i-1},m_i}}. $$
  By Lemma \ref{lemma-HT} (iii) and (iv), the set $\Lambda_i$ is a spectrum of $\mu_{m_i}$.  Since $0\in \mbf{L}_{m_{i-1}, m_i}$ and $0\in \Lambda_{i-1}$, it is clear that $0\in \Lambda_i$ and $\Lambda_{i-1} \sse \Lambda_i$.

Write $$\Lambda = \bigcup_{i=1}^\f \Lambda_i.$$
Next we show that $\Lambda$ is a spectrum of $\mu$. By Theorem~\ref{criterion} (ii), it is equivalent to show that
for each $ \xi \in \R$,
$$Q(\xi) = \sum_{\lambda \in \Lambda} |\wh{\mu}(\lambda + \xi)|^2=1.$$

  For each $\xi \in \R$, since $\Lambda_i$ is a spectrum of $\mu_{m_i}$, by Theorem~\ref{criterion} (ii), we have
  \begin{equation}\label{n_i=1}
  \sum_{\lambda \in \Lambda_i} |\wh{\mu}_{m_i}(\lambda + \xi)|^2 =1.
\end{equation}
  It follows that
  \begin{align*}
  \sum_{\lambda \in \Lambda_i} \big |\wh{\mu}(\lambda + \xi)\big|^2 &= \sum_{\lambda \in \Lambda_i} \big|\wh{\mu}_{m_i}(\lambda + \xi)\big|^2 \big|\wh{\mu}_{>m_i}(\lambda + \xi)\big|^2 \\
  &\le \sum_{\lambda \in \Lambda_i} \big|\wh{\mu}_{m_i}(\lambda + \xi)\big|^2 \\
  &\le 1.
 \end{align*}
Letting $i \to \f$, we obtain that
\begin{equation}\label{ineqQ}
Q(\xi) \le 1,
\end{equation}
for all  $\xi \in \R$.

  Fix $\xi\in \R$.  For each $\lambda \in \Lambda$, we define
  $$f(\lambda) = |\wh{\mu}(\lambda+ \xi)|^2,$$
and for $i \ge 1$,
  $$ f_i(\lambda) =
     \begin{cases}
       |\wh{\mu}_{m_i}(\xi+\lambda)|^2, & \mbox{if } \lambda \in \Lambda_i; \\
       0, & \mbox{if } \lambda \in \Lambda \sm \Lambda_i.
     \end{cases}
  $$
  For each $\lambda \in \Lambda$, there exists  $i_0 \ge 1$ such that $\lambda \in \Lambda_i$ for $i \ge i_0$, and it follows that $$ \lim_{i \to \f} f_{i}(\lambda) =  \lim_{i \to \f} |\wh{\mu}_{m_i}(\xi+\lambda)|^2 =f(\lambda). $$
  Choose an integer $i_0 \ge 1$ sufficiently large such that for $i > i_0$,
  \begin{equation}\label{ineqxi}
  \left| \frac{\xi}{N_1 N_2 \cdots N_{m_{i}}} \right| < \frac{\delta}{2}.
\end{equation}
  For each $\lambda \in \Lambda_i$ where $i > i_0$, by~\eqref{defLam}, we have that
 $$\lambda= \lambda_1 + (N_1 N_2 \cdots N_{m_{i-1}})\lambda_2 + (N_1 N_2 \cdots N_{m_i}) k_{\lambda_2,i},$$
 where $\lambda_1 \in \Lambda_{i-1}$ and $\lambda_2\in \mbf{L}_{m_{i-1}, m_i}$.
  By ~\eqref{ineqla} and \eqref{ineqxi}, we have that
  $$\Big|\frac{\lambda_1 + \xi}{N_1 N_2 \cdots N_{m_i}}\Big| < \delta.$$
  It follows from \eqref{lowerbound} that
  \begin{align*}
    f(\lambda) & = |\wh{\mu}(\lambda + \xi)|^2 = \big|\wh{\mu}_{m_i}(\lambda + \xi)\big|^2 \big|\wh{\mu}_{>m_i}(\lambda + \xi)\big|^2 \\
    &=\big|\wh{\mu}_{m_i}(\lambda + \xi)\big|^2 \left| \wh{\nu}_{>m_i}\left( \frac{\lambda + \xi}{N_1 N_2 \cdots N_{m_i}} \right) \right|^2 \\
    & = \big|\wh{\mu}_{m_i}(\lambda + \xi)\big|^2 \left| \wh{\nu}_{>m_i}\left( \frac{\lambda_2}{N_{m_{i-1}+1} \cdots N_{m_i}} + \frac{\lambda_1 + \xi}{N_1 N_2 \cdots N_{m_i}} + k_{\lambda_2,i} \right) \right|^2 \\
    & \ge \ep^2 f_i(\lambda).
  \end{align*}
  Therefore, for $i > i_0$, we have $$f_i(\lambda) \le \ep^{-2} f(\lambda),$$
  for all $\lambda \in \Lambda.$  Let $\rho$ be the counting measure on the set $\Lambda$.  We have that
  $$ \int_\Lambda f(\lambda) \D \rho(\lambda) = \sum_{\lambda \in \Lambda} |\wh{\mu}(\lambda + \xi)|^2 = Q(\xi) .
  $$
By ~\eqref{ineqQ}, $f(\lambda)$ is integrable with respect to the counting measure $\rho$.
Applying the dominated convergence theorem and by \eqref{n_i=1}, we obtain that
\begin{align*}
Q(\xi) &= \lim_{i \to \f} \int_\Lambda f_i(\lambda) \D \rho(\lambda)  \\
&= \lim_{i \to \f} \sum_{\lambda \in \Lambda_i} |\wh{\mu}_{m_i}(\lambda + \xi)|^2   \\
&= 1.
\end{align*}
Hence, by Theorem~\ref{criterion} (ii), the family $\set{e_\lambda: \lambda \in \Lambda}$ is an orthonormal basis in $L^2(\mu)$, and $\mu$ is a spectral measure with a spectrum $\Lambda \subseteq \Z$.
\end{proof}

\subsection{Admissibility}
The equi-positivity condition is rather technical, and it is not easy to check.
Therefore, the admissible family is introduced to guarantee the existence of equi-positive family. For $\mu \in \mcal{P}(\R^d)$, we write
\begin{equation}\label{integral-periodic-zero}
  \mcal{Z}(\mu) = \set{\xi \in \R^d: \wh{\mu}(\xi+k) = 0 \text{ for all } k \in \Z^d},
\end{equation}
for the {\it integral periodic zero set} of Fourier transform of $\mu$.
For $\Phi \sse \mcal{P}(\R^d)$, we write
$$
\mathrm{cl}(\Phi) = \set{ \mu \in \mcal{P}(\R^d): \text{ there exists $\set{\mu_n} \sse \Phi$ such that $\mu_n$ converges weakly to $\mu$} },$$
for the closure of $\Phi$ with respect to the weak topology on  $\mcal{P}(\R^d)$.

\begin{definition}
  Let $\Phi \sse \mcal{P}(\R^d)$. We call $\Phi$ an admissible family if $\mcal{Z}(\mu) = \emptyset$ for every $\mu\in \mathrm{cl}(\Phi)$.
\end{definition}

The following theorem shows that the admissibility implies the equi-positivity under the tightness condition. The proof is similar to the one in \cite{An-Fu-Lai-2019}.
\begin{theorem}\label{admissible-to-equipositive}
  Suppose that $\Phi \sse \mcal{P}(\R^d)$ is tight.
  If $\Phi$ is an admissible family, then $\Phi$ is equi-positive.
\end{theorem}
\begin{proof}
  We first claim that for each $x\in [0,1]^d$, there exists $\ep_x>0$ such that for all $\mu\in \Phi$ we have $$ \sup\set{ |\wh{\mu}(x+k)|: \;  k \in \Z^d } > \ep_x. $$
  Suppose that we may find $x_0 \in [0,1]^d$ such that for each $n \ge 1$ there exists $\mu_n \in \Phi$ satisfying $$ \sup\set{ |\wh{\mu}_n(x_0+k)|: \; k \in \Z^d } \le \frac{1}{n}. $$
  Since $\{ \mu_n \} \sse \Phi$ and $\Phi$ is tight, by Theorem \ref{weak-compactness}, there exists a subsequence $\{ \mu_{n_j}\}$ which converges weakly to a Borel probability measure $\mu \in \mcal{P}(\R^d)$.
  It follows from Theorem \ref{equivalent-condition-weak-converge} that for every $k \in \Z^d$ we have $$ \wh{\mu}(x_0+k) = \lim_{j \to \f} \wh{\mu}_{n_j}(x_0+k) = 0. $$
  This implies $x_0\in \mcal{Z}(\mu)$.
  It contradicts with $\mcal{Z}(\mu)=\emptyset$ since $\Phi$ is admissible and $\mu \in \mathrm{cl}(\Phi)$ .

  Therefore, for $x\in [0,1]^d$ and $\mu \in \Phi$, there exists an integral vector $k_{x,\mu}\in \Z^d$ such that $$ |\wh{\mu}(x+k_{x,\mu})| > \ep_x . $$
  By Lemma \ref{equicontinuous}, the family $\set{\wh{\mu}: \mu \in \Phi}$ is equicontinuous.
  Thus, for each $\ep_x>0$, there exists $\delta_x>0$ such that
  $$ |\wh{\mu}(y_1) -\wh{\mu}(y_2)| <\frac{\ep_x}{2}, $$ for all $ \mu\in \Phi$ and all $|y_1 - y_2|<\delta_x$.
  It follows that for $\mu \in \Phi$ and $|y| < \delta_x$,
  \begin{equation}\label{ineqmue}
     |\wh{\mu}(x+y+k_{x,\mu})| \ge |\wh{\mu}(x+k_{x,\mu})| - |\wh{\mu}(x+y+k_{x,\mu}) - \wh{\mu}(x+k_{x,\mu})| > \frac{\ep_x}{2}.
  \end{equation}
  Since $[0,1]^d$ is compact and $$ [0,1]^d \sse \bigcup_{x\in [0,1]^d} U(x,\delta_x/2), $$
  there exist finitely many $x_1, x_2, \cdots, x_p \in [0,1]^d$ such that $$ [0,1]^d \sse \bigcup_{j=1}^p U(x_j,\delta_{x_j}/2). $$
  Let $ \ep = \min \set{ \ep_{x_j}/2: j=1,2,\cdots, p}$ and $\delta = \min\{ \delta_{x_j}/2: j=1,2,\cdots,p \}$.

  For each $x\in [0,1)^d \sm \set{{\bf 0}}$, we may find some $1\le j \le p$ such that $x\in U(x_j,\delta_{x_j}/2) $.
  For $\mu \in \Phi$ and $|y| < \delta$, we have $| x- x_j + y | < \delta_{x_j}$, and by \eqref{ineqmue} it follows that
  \[ |\wh{\mu}(x+y+k_{x_j,\mu})| = \left|\wh{\mu}\big( x_j + (x-x_j+y)+k_{x_j,\mu}\big) \right| \ge \frac{\ep_{x_j}}{2} \ge \ep. \]
  Thus, we set $k_{x,\mu}=k_{x_j,\mu}$.

For $x={\bf 0}$, noting that  $|y| < \delta$ implies $|y| < \delta_{x_1}$, we have that
 $$ |\wh{\mu}(y)| \ge |\wh{\mu}({\bf 0})| - |\wh{\mu}({\bf 0}) -\wh{\mu}(y)|  \ge 1- \frac{\ep_{x_1}}{2} \ge \ep ,$$
for $\mu \in \Phi$ and $|y| < \delta$. Therefore, we set $k_{x,\mu}={\bf 0}$ for $x={\bf 0}$, and the conclusion holds.
\end{proof}

Next theorem is a direct consequence of Theorem~\ref{general-result} and Theorem \ref{admissible-to-equipositive}.
\begin{theorem}\label{thm-adm}
  Let $\set{(N_n, B_n , L_n): n \ge 1}$ be a sequence of Hadamard triples in $\R$. Let the probability measure $\mu$ be the weak limit of $\mu_n$ given by \eqref{mu-n}.
  If there exists a subsequence of positive integers $\{ n_j \}$ such that the family $\{ \nu_{>n_j} \}$ is tight and admissible, then $\mu$ is a spectral measure with a spectrum in $\Z$.
\end{theorem}

\section{Proof of Theorem \ref{main-result}}\label{sec_pf}

For a finite subset $B \sse \Z$, we set
$$
M_{B}(\xi) = \frac{1}{\# B} \sum_{b \in B} e^{-2\pi i b \xi}.
$$
In fact, $M_B$ is the Fourier transform of the discrete measure $\delta_B$.
If $(N,B,L)$ is a Hadamard triple in $\R$, then the set $L$ is a spectrum of the discrete measure $\delta_{N^{-1}B}$.  Since $\wh{\delta}_{N^{-1} B}(\xi) = M_B(\xi/N)$, by Theorem \ref{criterion}, we have that for all $\xi \in \R$,
$$
\sum_{\ell\in L} \left| M_{B}\left( \frac{\xi+\ell}{N}\right) \right|^2 = 1.
$$

Recall that $\set{(N_j, B_j, L_j): 1 \le j \le m}$ is a set of finitely many Hadamard triples in $\R$, and for $ \omega\in \Omega$,
$$ \mu_\omega = \delta_{N_{\omega_1}^{-1} B_{\omega_1}} * \delta_{(N_{\omega_1} N_{\omega_2})^{-1} B_{\omega_2}} * \cdots * \delta_{(N_{\omega_1} N_{\omega_2}\cdots N_{\omega_k})^{-1} B_{\omega_k} }* \cdots. $$
 The Fourier transform of $ \mu_\omega$ is given by
 $$
 \wh{\mu}_{\omega}(\xi) = \wh{\mu}_{\sigma^n(\omega)}\big( (N_{\omega_1} N_{\omega_2} \cdots N_{\omega_n})^{-1}\xi \big) \prod_{k=1}^n M_{B_{\omega_k}}\big( (N_{\omega_1} N_{\omega_2}\cdots N_{\omega_k})^{-1}\xi\big).$$

For a function $f:\R \to \C$, we denote the zero set of $f$ by
$$\mcal{O}(f)=\set{x\in \R: f(x) =0} .
$$
The next lemma indicates that there are only finitely many zero points of  $\{\wh{\mu}_\omega\}$ in every finite interval.
\begin{lemma}\label{lemma-finite-set}
  For every $h >0$, the set
  \begin{equation}\label{set-finite-1}
    [-h,h] \cap \left( \bigcup_{\omega \in \Omega} \mcal{O}(\wh{\mu}_\omega) \right)
  \end{equation}
  is finite.
\end{lemma}
\begin{proof}
  Since there are only finitely many Hadamard triples, we may find a compact subset $K \sse \R$ such that $\mu_\omega \in \mcal{P}(K)$ for all $\omega \in \Omega$. Therefore, the family $\{\mu_\omega: \omega \in \Omega\}$ is tight. By Lemma \ref{equicontinuous}, we have that $\{\wh{\mu}_\omega: \omega \in \Omega\}$ is equicontinuous.
  Since $\wh{\mu}_\omega(0)=1$ for every $\omega \in \Omega$, there exists $\delta>0 $ such that for all $ \omega \in \Omega$ and all $|y|<\delta$, $$|\wh{\mu}_\omega(y)| \ge 1/2.$$

  Suppose that $\wh{\mu}_{\omega}(\xi)=0$. We choose a sufficiently large integer $n$ such that $$|(N_{\omega_1} \cdots N_{\omega_n})^{-1}\xi| < \delta.$$
 Note that
 $$ \wh{\mu}_{\omega}(\xi) = \wh{\mu}_{\sigma^n(\omega)}\big( (N_{\omega_1} \cdots N_{\omega_n})^{-1}\xi \big) \prod_{k=1}^n M_{B_{\omega_k}}\big( (N_{\omega_1} \cdots N_{\omega_k})^{-1}\xi\big).$$
 Thus, we have $M_{B_{\omega_k}}\big( (N_{\omega_1} \cdots N_{\omega_k})^{-1}\xi\big)=0$ for some $k \ge 1$.
 Therefore, $\wh{\mu}_{\omega}(\xi)=0$ if and only if there exists $k \ge 1$ such that
$$
M_{B_{\omega_k}}\big( (N_{\omega_1} \cdots N_{\omega_k})^{-1}\xi\big) =0.
$$
It follows that
  \begin{align*}
    \mcal{O}(\wh{\mu}_\omega)&  = \bigcup_{k=1}^\f N_{\omega_1} N_{\omega_2} \cdots N_{\omega_k} \mcal{O}(M_{B_{\omega_k}}) \\
    & \sse \bigcup_{j=1}^m \bigcup_{k=1}^\f N_{\omega_1} N_{\omega_2} \cdots N_{\omega_k} \mcal{O}(M_{B_j}) \\
    & \sse \bigcup_{j=1}^m \bigcup_{k_1 =0}^\f \bigcup_{k_2=0}^\f \cdots \bigcup_{k_m =0}^\f N_1^{k_1} N_2^{k_2} \cdots N_m^{k_m} \mcal{O}(M_{B_j}).
  \end{align*}

Therefore, it suffices to show that for every $1\le j \le m$, the set
  \begin{equation}\label{set-finite-2}
    [-h, h] \cap \left( \bigcup_{k_1 =0}^\f \bigcup_{k_2=0}^\f \cdots \bigcup_{k_m =0}^\f N_1^{k_1} N_2^{k_2} \cdots N_m^{k_m} \mcal{O}(M_{B_j}) \right)
  \end{equation}
  is finite.
  Since $M_{B_j}$ is extendable to an entire function on the complex plane, the set $\mcal{O}(M_{B_j})$ is a discrete set in $\R$.
  Noting that $0 \not \in \mcal{O}(M_{B_j})$,  we may find  $\delta_j >0$ such that
  $$[-\delta_j,\delta_j] \cap \mcal{O}(M_{B_j}) = \emptyset.$$
If $k_1 + k_2 +\cdots+ k_m > \log(h/\delta_j) / \log 2$, then we have
$$
|N_1^{k_1} N_2^{k_2} \cdots N_m^{k_m}| \delta_j \ge 2^{k_1 +k_2 + \cdots+ k_m} \delta_j  > h.
$$
  It follows that $[-h,h] \cap  \big( N_1^{k_1} N_2^{k_2} \cdots N_m^{k_m} \mcal{O}(M_{B_j}) \big) =\emptyset $.
  Therefore, there are  only finitely many $m$-tuples $(k_1, k_2, \cdots, k_m)$ such that
  $$[-h,h] \cap  \big( N_1^{k_1} N_2^{k_2} \cdots N_m^{k_m} \mcal{O}(M_{B_j}) \big) \ne \emptyset. $$
  Since $\mcal{O}(M_{B_j})$ is a discrete set in $\R$,
  we conclude that the set in (\ref{set-finite-2}) is finite, and this completes the proof.
\end{proof}

\begin{proposition}\label{periodic-zero-set}
  Suppose that $\gcd(B_j - B_j) =1$ for $1\le j \le m$.
  Then we have $\mcal{Z}(\mu_\omega) = \emptyset$ for every $\omega \in \Omega$.
\end{proposition}
\begin{proof}
  For $a\in \R$, recall that $\wh{\mu*\delta_a}(\xi) = e^{-2\pi i a \xi} \wh{\mu}(\xi)$, and we have that  $\mcal{Z}(\mu) = \mcal{Z}(\mu*\delta_a)$.
  This implies that the translation does not change the integral period zero set of Fourier transform of a measure.
Let $\wt{B}_j$ be the translation of $B_j$ such that $0\in \wt{B}_j $ for $1 \le j \le m$.
Let $\wt{L}_j$ be the subset of $ \set{0,1,\cdots, |N_j|-1}$ such that $\wt{L}_j \equiv L_j \pmod{ N_j\Z}$ for $1\le j \le m$. By Lemma~\ref{lemma-HT} (i) and (iii),  $(N_j,\wt{B}_j ,\wt{L}_j)$ is still a Hadamard triple for $1 \le j \le m$. Let $\mu_\omega$ and $\wt{\mu}_\omega$ be the random convolution generated by $\{(N_j, B_j , L_j), 1\leq j \leq m\}$ and $\{(N_j,\wt{B}_j ,\wt{L}_j), 1\leq j \leq m\}$, respectively.
Clearly  we have that $\mcal{Z}(\mu_\omega) = \mcal{Z}(\wt{\mu}_\omega) $.
Therefore, for simplicity, we assume that $0\in B_j$ and  $L_j \sse \set{0,1,\cdots, |N_j|-1}$ for $1 \le j \le m$.
We prove this proposition by contradiction.

Suppose that there exists $\omega \in \Omega$ such that $\mcal{Z}(\mu_\omega) \ne \emptyset$.  For $\ell \in L_j$, we write
$$
\tau_{\ell,j}(x) = N_j^{-1}(x+\ell).
$$
 Arbitrarily choose $\xi_0 \in \mcal{Z}(\mu_\omega)$ and set $Y_0 = \set{\xi_0}$. For $n \ge 1$, we define
$$ Y_n = \set{ \tau_{\ell,\omega_n}(\xi) : \;\xi \in Y_{n-1}, \; \ell \in L_{\omega_n}, \; M_{B_{\omega_n}}\big( \tau_{\ell,\omega_n}(\xi) \big) \ne 0 }.
$$

First, we show that for each $n \ge 1$,
$$ \# Y_{n-1} \le \# Y_n. $$
Since for each $ \xi \in Y_{n-1}$,
$$ \sum_{\ell \in L_{\omega_n} } \left| M_{B_{\omega_n}}\big( \tau_{\ell,\omega_n}(\xi) \big) \right|^2 =1,$$
  there exists at least one element $\ell \in L_{\omega_n}$ such that $M_{B_{\omega_n}}\big( \tau_{\ell,\omega_n}(\xi) \big) \ne 0$.
  On the other hand, for distinct $\ell_1 \ell_2 \cdots \ell_n \ne \ell_1' \ell_2' \cdots \ell_n'$ where $\ell_j, \ell_j' \in L_{\omega_j}$, we have
  $$\tau_{\ell_n,\omega_n} \circ \cdots \circ \tau_{\ell_2,\omega_2} \circ \tau_{\ell_1,\omega_1} (\xi_0) \ne \tau_{\ell_n',\omega_n} \circ \cdots \circ \tau_{\ell_2',\omega_2} \circ \tau_{\ell_1',\omega_1}(\xi_0).$$
  Otherwise, $$ \frac{ \xi_0 + \ell_1 + N_{\omega_1} \ell_2 + \cdots + N_{\omega_1} \cdots N_{\omega_{n-1}} \ell_n }{ N_{\omega_1} N_{\omega_2} \cdots N_{\omega_n} } = \frac{ \xi_0 + \ell_1' + N_{\omega_1} \ell_2' + \cdots + N_{\omega_1} \cdots N_{\omega_{n-1}} \ell_n' }{N_{\omega_1} N_{\omega_2} \cdots N_{\omega_n}}, $$
  that is, $$  \ell_1 + N_{\omega_1} \ell_2 + \cdots + N_{\omega_1} N_{\omega_2} \cdots N_{\omega_{n-1}} \ell_n = \ell_1' + N_{\omega_1} \ell_2' + \cdots + N_{\omega_1} N_{\omega_2} \cdots N_{\omega_{n-1}} \ell_n'.$$
  Let $j_0 = \min\set{1 \le j \le n: \ell_j \ne \ell_j'}$.
  Then we have $N_{\omega_{j_0}} \mid \ell_{j_0} - \ell_{j_0}'$.
  But by Lemma \ref{lemma-HT} (ii), the elements in $L_{\omega_{j}}$ are in distinct residue classes modulo $N_{\omega_j} \Z$. This leads to a contradiction.
  Therefore, we conclude that $\# Y_{n-1} \le \# Y_n$ for $n \ge 1$.

 Next, we prove that for each $n \ge 0$,
 $$ Y_n \sse \mcal{Z}(\mu_{\sigma^n(\omega)}) .$$
For $n=0$, it is clear that  $Y_0 \sse \mcal{Z}(\mu_\omega)$.  For $n \ge 1$, we assume that $Y_{n-1} \sse \mcal{Z}(\mu_{\sigma^{n-1}(\omega)})$.
For each $\tau_{\ell,\omega_n}(\xi) \in Y_{n}$ where $\xi \in Y_{n-1}$, $\ell \in L_{\omega_n}$, and $M_{B_{\omega_n}}\big( \tau_{\ell,\omega_n}(\xi) \big) \ne 0$, we have that for every $k \in \Z$,
  \begin{align*}
    0 & = \wh{\mu}_{\sigma^{n-1}(\omega)}(\xi + \ell + N_{\omega_n} k) \\
    &= M_{B_{\omega_{n}}}\left( \frac{\xi + \ell}{ N_{\omega_n} } +k \right) \wh{\mu}_{\sigma^{n}(\omega)} \left( \frac{\xi + \ell}{N_{\omega_n}} +k \right)\\
     & = M_{B_{\omega_{n}}}\big( \tau_{\ell,\omega_n}(\xi) \big) \wh{\mu}_{\sigma^{n}(\omega)} \big( \tau_{\ell,\omega_n}(\xi) +k \big),
  \end{align*}
  where the last equality follows from the integral periodicity of $M_{B_{\omega_n}}$.
  Since $M_{B_{\omega_{n}}}\big( \tau_{\ell,\omega_n}(\xi) \big) \ne 0$, we have that
  $$ \wh{\mu}_{\sigma^{n}(\omega)} \big( \tau_{\ell,\omega_n}(\xi) +k \big) =0 ,
  $$ for all $k \in \Z$.
This implies that $\tau_{\ell,\omega_n}(\xi) \in \mcal{Z}(\mu_{\sigma^{n}(\omega)})$. Thus, we have $Y_{n} \sse \mcal{Z}(\mu_{\sigma^{n}(\omega)})$.
  By induction, we obtain that $Y_n \sse \mcal{Z}(\mu_{\sigma^n(\omega)})$ for all $n \ge 0$.

Finally,  we use the increasing cardinality of $Y_n$ and Lemma \ref{lemma-finite-set} to deduce a contradiction.
For every $\xi \in Y_n$, by the definition of $Y_n$, we write $\xi$  as
\begin{align*}
\xi &= \tau_{\ell_n,\omega_n} \circ \cdots \circ \tau_{\ell_2,\omega_2} \circ \tau_{\ell_1,\omega_1} (\xi_0)  \\
 &= \frac{ \xi_0 + \ell_1 + N_{\omega_1} \ell_2 + \cdots + N_{\omega_1} \cdots N_{\omega_{n-1}} \ell_n }{ N_{\omega_1} N_{\omega_2} \cdots N_{\omega_n} } .
\end{align*}
  Since $|N_{\omega_j}|\geq 2$ and $0\le \ell_j < |N_{\omega_j}|$, we have that
  \begin{align*}
    |\xi| &\le \frac{|\xi_0|}{2^n} + \frac{1}{2^{n-1}} + \frac{1}{2^{n-2}} + \cdots + 1 \\
    &\le |\xi_0| + 2.
  \end{align*}
  Let $h= |\xi_0| + 2$. Then $Y_n \sse [-h , h]$ for every $n \ge 1$.
  It follows that
  \begin{align*}
    Y_n &\sse [-h,h] \cap \mcal{Z}(\mu_{\sigma^n(\omega)})  \\
    &\sse [-h,h] \cap \left( \bigcup_{\eta \in \Omega} \mcal{O}(\wh{\mu}_\eta) \right).
  \end{align*}
  By Lemma \ref{lemma-finite-set}, the set $[-h,h] \cap \left( \bigcup_{\eta \in \Omega} \mcal{O}(\wh{\mu}_\eta) \right)$ is  finite. Since the sequence of $\# Y_n$ is increasing, there exists $n_0 \ge 1$ such that $\# Y_n = \# Y_{n+1}$ for all $n>n_0$.
Given $n>n_0$, for every $\xi \in Y_{n}$, there exists a unique $\ell_0 \in L_{\omega_{n+1}}$ such that  $M_{B_{\omega_{n+1}}}\big( \tau_{\ell_0,\omega_{n+1}}(\xi) \big) \ne 0$.
Recall that
$$\sum_{\ell \in L_{\omega_{n+1}}} \left| M_{B_{\omega_{n+1}}}\big( \tau_{\ell,\omega_{n+1}}(\xi) \big) \right|^2 =1.
$$
This implies that  $$ \left| M_{B_{\omega_{n+1}}}\big( \tau_{\ell_0,\omega_{n+1}}(\xi) \big)\right| = \left| \frac{1}{\# B_{\omega_{n+1}}} \sum_{b\in B_{\omega_{n+1}}} e^{-2\pi i b \tau_{\ell_0,\omega_{n+1}}(\xi) } \right| =1. $$
  Since $0 \in B_{\omega_{n+1}}$, we have that $b \tau_{\ell_0,\omega_{n+1}}(\xi) \in \Z$ for $b\in B_{\omega_{n+1}}$.
  Since  $\gcd(B_{\omega_{n+1}} - B_{\omega_{n+1}}) =1$  and $0\in B_{\omega_{n+1}}$, it is clear that $\gcd(B_{\omega_{n+1}})=1$. Then there exist integers $m_b \in \Z$ for $b \in B_{\omega_{n+1}}$ such that $\sum_{b\in B_{\omega_{n+1}}} m_b b =1$.
  This implies that $$ \tau_{\ell_0,\omega_{n+1}}(\xi) = \sum_{b\in B_{n+1}} m_b b\tau_{\ell_0,\omega_{n+1}}(\xi) \in \Z. $$
  Since $\wh{\mu}_{\sigma^{n+1}(\omega)}(0) =1$, we have that $\mcal{Z}(\mu_{\sigma^{n+1}(\omega)}) \cap \Z = \emptyset$,
which contradicts with the fact
$$\tau_{\ell_0,\omega_{n+1}}(\xi) \in Y_{n+1} \sse \mcal{Z}(\mu_{\sigma^{n+1}(\omega)}).
$$
Therefore, the conclusion holds.
\end{proof}

\begin{proposition}\label{prop_closure}
  Let $\Phi=\set{ \mu_\omega: \omega \in \Omega }$. Then we have $\mathrm{cl}(\Phi)=\Phi$.
\end{proposition}
\begin{proof}
  For $\omega \in \Omega$ and $n \ge 1$, we write $$ \mu_{\omega,n} = \delta_{N_{\omega_1}^{-1} B_{\omega_1}} * \delta_{(N_{\omega_1} N_{\omega_2})^{-1} B_{\omega_2}} * \cdots * \delta_{(N_{\omega_1} N_{\omega_2}\cdots N_{\omega_n})^{-1} B_{\omega_n} },$$
 and $\mu_\omega = \mu_{\omega,n} * \mu_{\omega,>n}$.
 Let $ h = 1 + \max\set{ |b|: b \in B_j, 1\le j \le m } $.
 It is clear that for each $\omega \in \Omega$ and each $n \ge 1$,
 $$ \mathrm{spt}(\mu_{\omega,n}) \sse [-(h-1), h-1], $$
and
$$\mathrm{spt}(\mu_{\omega,>n}) \sse [-2^{-n} h, 2^{-n} h]. $$
Arbitrarily choose a sequence $\{ \mu_{\omega(j)} \}_{j=1}^\infty $ of probability measures in $\Phi$ which converges weakly to a probability measure $\mu \in \mcal{P}(\R)$, where $\{\omega(j)\}_{j=1}^\infty$ is a sequence in $\Omega$ and we write $\omega(j)=\omega_1 (j)\omega_2 (j)\cdots $.
Since $\Omega$ is compact, $\{\omega(j)\}_{j=1}^\infty$ has a convergent subsequence. For simplicity, we assume that $\{\omega(j)\}_{j=1}^\infty$ converges to $\eta=(\eta_k)_{k=1}^\f \in \Omega$.
Next,  we prove that $\mu_{\omega(j)}$ converges weakly to $\mu_{\eta}$.

Fix $f \in C_b(\R)$. For each $\ep>0$, since $f$ is uniformly continuous on the interval $[-h,h]$, there exists  $0<\delta < 1$ such that for all $x,y \in [-h,h]$ with $|x-y| <\delta$, we have $$ |f(x) - f(y)| < \frac{\ep}{2}. $$
  Then we choose a sufficiently large integer $n$ such that $2^{-n} h < \delta$.
For every $\omega \in \Omega$, we have that
  \begin{align*}
    \int_{\R} f(x) \D \mu_{\omega}(x) & = \int_{\R} f(x) \D \mu_{\omega,n}* \mu_{\omega,>n} (x) \\
    &= \int_{\R^2} f(x+y) \D \mu_{\omega,n}\times \mu_{\omega,>n} (x,y) \\
     & = \int_{\R} \int_{\R} f(x+y) \D\mu_{\omega,>n}(y) \D \mu_{\omega,n}(x).
  \end{align*}
Hence, by the uniform continuity of $f$ on $[-h,h]$, we get that
  \begin{eqnarray*}
    && \left| \int_{\R} f(x) \D \mu_{\omega}(x) - \int_{\R} f(x) \D \mu_{\omega,n}(x) \right| \\
    &=& \left| \int_{\R} \int_{\R} \big( f(x+y)- f(x) \big) \D\mu_{\omega,>n}(y) \D \mu_{\omega,n}(x) \right| \\
     &\le&  \int_{-(h-1)}^{h-1} \int_{-2^{-n}h}^{2^{-n}h} \left| f(x+y)- f(x) \right| \D\mu_{\omega,>n}(y) \D \mu_{\omega,n}(x) \\
    &\le& \frac{\ep}{2}.
  \end{eqnarray*}
  Since the sequence $\set{\omega(j)}$ converges to $\eta$, there exists an integer $j_0\ge 1$ such that for all $j \ge j_0$,
  $$ \omega_1(j) \omega_2(j) \cdots \omega_n(j) = \eta_1 \eta_2 \cdots \eta_n,
  $$
  So for $j\ge j_0$, we have that
  $$ \left| \int_{\R} f \D \mu_{\omega(j)} - \int_{\R} f \D \mu_{\eta} \right| \le
    \left| \int_{\R} f \D \mu_{\omega(j)} - \int_{\R} f \D \mu_{\omega(j),n} \right| + \left| \int_{\R} f \D \mu_{\eta} - \int_{\R} f \D \mu_{\eta,n} \right| \le \ep. $$
  It follows that $$ \lim_{j \to \f} \int_{\R} f \D \mu_{\omega(j)} = \int_{\R} f \D \mu_{\eta},$$
for all $ f \in C_b(\R)$. Therefore, $\mu_{\omega(j)}$ converges weakly to $\mu_{\eta}$.

By the uniqueness of weak limit, we have that $\mu = \mu_\eta \in \Phi$, and the conclusion holds.
\end{proof}

\begin{corollary}\label{corollary}
  Suppose that $\gcd(B_j - B_j) =1$ for $1\le j \le m$.
  Then the family $\Phi=\set{ \mu_\omega: \omega \in \Omega }$ is admissible, and thus, is equi-positive.
\end{corollary}
\begin{proof}
By Proposition~\ref{periodic-zero-set} and Proposition~\ref{prop_closure}, we have that $\Phi=\set{ \mu_\omega: \omega \in \Omega }$ is admissible. Since there are only finitely many Hadamard triples, we may find a compact subset $K \sse \R$ such that $\mu_\omega \in \mcal{P}(K)$ for all $\omega \in \Omega$. This implies that $\Phi=\set{ \mu_\omega: \omega \in \Omega }$ is tight.
By Theorem~\ref{admissible-to-equipositive}, we have that $\Phi=\set{ \mu_\omega: \omega \in \Omega }$ is equi-positive.
\end{proof}

Finally, we are ready to prove the main theorem.
\begin{proof}[Proof of Theorem \ref{main-result}]
  Fix a sequence of positive integers $\set{n_k}$ and $\omega \in \Omega$,
  and write $$\mu = \mu_{\omega,\set{n_k}} = \delta_{N_{\omega_1}^{-n_1} B_{\omega_1}} * \delta_{N_{\omega_1}^{-n_1} N_{\omega_2}^{-n_2} B_{\omega_2}} * \cdots * \delta_{N_{\omega_1}^{-n_1} N_{\omega_2}^{-n_2} \cdots N_{\omega_k}^{-n_k} B_{\omega_k} }* \cdots . $$
  Note that $\mu$ is the corresponding infinite convolution generated by the sequence of Hadamard triples
  $\set{\big( (N_{\omega_k})^{n_k}, B_{w_k}, (N_{\omega_k})^{n_k-1} L_{\omega_k} \big): k \ge 1}$.
  Recall the notations in (\ref{mu-infinite-convolution}) and (\ref{nu-large-than-n}), and we have $$ \nu_{>k} = \delta_{ N_{\omega_{k+1}}^{-n_{k+1}} B_{\omega_{k+1}} } * \delta_{ N_{\omega_{k+1}}^{-n_{k+1}} N_{\omega_{k+2}}^{-n_{k+2}} B_{\omega_{k+2}} } * \cdots.$$
  Let $\eta = \omega_1^{n_1} \omega_2^{n_2} \omega_3^{n_3}\cdots $, that is, $\eta_{j} = \omega_k$ for $n_1 + \cdots + n_{k-1} < j \le n_1 +\cdots + n_k$.
  By Lemma \ref{convolution-weak-convergence}, we factor the measure $\mu_\eta$ into $\mu_\eta = \mu * \rho$ for some $\rho\in \mcal{P}(\R)$. (The measure $\rho$ is also an infinite convolution, but we do not need the precise formula of $\rho$ in our proof.)
  It is also easy to check that $\mu_{\sigma^{n_1 + \cdots + n_k}(\eta)} = \nu_{>k} * \rho_k$ for some $\rho_k\in \mcal{P}(\R)$.
  Thus, for each $\xi \in \R$, we have
  \begin{equation}\label{eq-4-3}
     |\wh{\mu}_{\sigma^{n_1 + \cdots + n_k}(\eta)}(\xi)| = |\wh{\nu}_{>k}(\xi) \wh{\rho}_k(\xi)| \le |\wh{\nu}_{>k}(\xi)|.
  \end{equation}
  By Corollary \ref{corollary}, the family $\{ \mu_{\sigma^{n_1 + \cdots + n_k}(\eta)} \}_{k=1}^\f$ is equi-positive.
  It follows from Definition \ref{def-equipositive} and (\ref{eq-4-3}) that the family $\{ \nu_{>k} \}_{k=1}^\f$ is also equi-positive.
  Therefore, by Theorem \ref{general-result}, $\mu = \mu_{\omega,\set{n_k}}$ is a spectral measure.
\end{proof}

\section*{Acknowledgements}
The authors thank to Prof. Xing-Gang He and Li-Xiang An for their helpful comments.

Wenxia Li is supported by NSFC No.~12071148,~11971079 and Science and Technology Commission of Shanghai Municipality (STCSM) No.~22DZ2229014.
Jun Jie Miao is supported by Science and Technology Commission of Shanghai Municipality (STCSM)  No.~22DZ2229014.
Zhiqiang Wang is supported by Fundamental Research Funds for the Central Universities No.~YBNLTS2023-016.
The authors would like to thank the referees for his/her many valuable comments and suggestions.

\end{document}